\documentclass [12pt,a4paper,reqno]{amsart}
\usepackage{amsfonts}
 \usepackage{amssymb}
\catcode`\@=11

 \def\AMSTeXfeatures{\Plainheads 
   \let\current@vert=\AMS@vert}

 \def\Plainheads{\sh@ftdiam=0.05em
   \getlabeldims
   \let\vshaftfill=\plnvsolidfill
   \let\hshaftfill=\plnhsolidfill
   \let\th@rhead=\plnrhead
   \let\th@lhead=\plnlhead
   \let\th@dnhead=\plndnhead
   \let\th@uphead=\plnuphead}
 
 \def\glet{\global\let}

 \def\LaTeXfeatures{\catcode`\@=11
   \ifx\@clnwd\undefined \nol@g
      \input ltxcode.tex \dol@g \fi
   \ltxheads \let\current@vert=\new@vert
   \providelto \catcode`\@=\active}

 \def\nol@g{\def\wlog{\edef\garbage}}
 \def\dol@g{\let\wlog=\wl@g} \let\wl@g=\wlog
 \nol@g 

 \newbox\ltobox
 \def\providelto{{\setbox\z@=
   \hbox{$\to$}\minharrlen=\wd\z@
   \global\setbox\ltobox=\hbox{$\activeat>>>$}}
   \def\lto{\mathrel{\copy\ltobox}}}

 \def\ltxheads{\sh@ftdiam=\@wholewidth
   \getlabeldims
   \let\vshaftfill= \ltxvsolidfill
   \let\hshaftfill=\ltxhsolidfill
   \let\th@rhead=\ltxrhead
   \let\th@lhead=\ltxlhead
   \let\th@dnhead=\ltxdnhead
   \let\th@uphead=\ltxuphead}
 {\catcode`\@=\active
   \gdef@#1{\csname #1\string@at\endcsname}
   \glet\activeat=@}
 \def\def@#1{\expandafter\def\csname #1@at\endcsname}

 \def@>#1>#2>{\@rrow R{#1}{#2}}
 \def@<#1<#2<{\@rrow L{#1}{#2}}
 \def@ V#1V#2V{\@rrow V{#1}{#2}}
 \def@ A#1A#2A{\@rrow A{#1}{#2}}
 \def@/#1/#2/#3/{\@rrow{#1}{#2}{#3}}
 \def@.{\ifodd\row\ifmmode\noharrow
     \else\leavevmode.\spacefactor3000 \fi
   \else\novarrow\fi}
 \def@={\ifodd\row\harrow\hequalfill{}{}%
   \else\varrow\vequalfill{}{}\fi}
 \def@:#1{\ifx=#1\harrow\deffill{}{}%
   \else\leavevmode\null:#1\fi}
 \def@|{\current@vert}
  \def\AMS@vert{\varrow\vequalfill{}{}}
  \def\new@vert#1|#2|{\ifodd\row
   \let\nextarrow\vertexvarrow
   \else\let\nextarrow\varrow\fi
   \nextarrow\vshaftfill{#1}{#2}}
 \def@-{\ifmmode\let\next\hl@ne
   \else\let\next\AMSatdash \fi \next}
  \def\hl@ne#1-#2-{\harrow\hshaftfill{#1}{#2}}
  \def\AMSatdash{\let\next\relax\leavevmode
    \def\next@{\ifx\next-%
      \def\next-{\futurelet\next\nextii@}%
     \else\def\next{\hbox{-}}\fi\next}%
    \def\nextii@{\ifx\next-\def\next-{\hbox{---}}%
      \else\def\next{\hbox{--}}\fi\next}%
    \futurelet\next\next@}
 \def@(#1){\tweenarrows{#1}}
 \def@[#1]{\setsp@n#1\relax\activeat}
 \def\fiberbox{\hbox{$\vcenter{\hr@le\hbox{\vr@le
   \kern1ex\vbox{\kern1.2ex}\vr@le}\hr@le}$}}
  \def\hr@le{\hrule height \sh@ftdiam}
  \def\vr@le{\vrule width \sh@ftdiam}
 \def@+#1+#2+#3+{\ifodd\row \harrow{#1}{#2}{#3}%
   \else \varrow{#1}{#2}{#3}\fi}

 \def\Dnarrfill{\vequalfill\Dnhe@d}
 \def\Uparrfill{\Uphe@d\vequalfill}
 
 \def\ontofill{\rtarrfill\kern-0.3em 
   \th@rhead\kern 0.3em} 

 \def\rtarrfill{\hshaftfill\th@rhead}
 \def\ltarrfill{\th@lhead\hshaftfill}
 \def\dnarrfill{\vshaftfill\th@dnhead}
 \def\uparrfill{\th@uphead\vshaftfill}
 \def\hequalfill{\plnhfill=}
 \def\deffill{:\plnhfill=}
 \def\plnvextfill#1{\setbox\z@
   \hbox{\the\textfont3 #1}%
   \dimen@=\dp\z@\advance\dimen@\ht\z@
   \copy\z@ \kern-\dimen@ 
   \cleaders\copy\z@ \vfill
   \kern-\dimen@ 
   \box\z@}
 \def\plnhfill#1{$\m@th\mkern-1.5mu\mathord#1\mkern-6mu
    \cleaders\hbox{$\mkern-2mu\mathord#1\mkern-2mu$}\hfill
    \mkern-6mu\mathord#1\mkern-1.5mu$}
 \def\vequalfill{\plnvextfill{\char'167}}
 \def\plnvsolidfill{\plnvextfill{\char'077}}
 \def\plnhsolidfill{\plnhfill-}
 \def\ltxhsolidfill{\leaders\hrule height\topofshaft depth\botofshaft
   \hfill}
 \def\ltxvsolidfill{\leaders\vrule width\sh@ftdiam\vfill}
 \def\hdashfill{\hd@sh\wd@sh
   \xleaders \hbox{\wd@sh\hd@sh\wd@sh}\hfill
   \wd@sh\hd@sh}
 \def\vdashfill{\vd@sh\wd@sh
   \xleaders \vbox{\wd@sh\vd@sh\wd@sh}\vfill
   \wd@sh\vd@sh}
 \def\dashed{\ifinmeasureCD\else
    \ifodd\row\option{\let\hshaftfill=\hdashfill}%
   \else\option{\let\vshaftfill=\vdashfill}\fi\fi}

 \newdimen\CDstrutht  \newdimen\CDstrutdp
 \newdimen\CDstrutlen \CDstrutlen=\CDstrutht
   \advance\CDstrutlen by \CDstrutdp

 \def\CDstrut{\vrule
   height \ifnum\row=1 \z@\else\CDstrutht \fi
   depth \ifnum\row=\numrows \z@ \else\CDstrutdp \fi
   width\z@}

 \newdimen\CDarrsurr \CDarrsurr=0.375em
 \newdimen\CDdashlen
 \newdimen\CDvarrlen \CDvarrlen=1.5\baselineskip
 \newdimen\minharrlen 
  \setbox\z@\hbox{$\longrightarrow$} \minharrlen=\wd\z@
 \newdimen\minCDharrlen \minCDharrlen=2.5em 
\newdimen \minc@lwd
\def\findminc@lwd{\minc@lwd=2\CDarrsurr
  \advance\minc@lwd\minCDharrlen}

 \newdimen\sh@ftdiam

 \newdimen\labelsurr \labelsurr=1.25 em

\newcount\sp@ncnt \sp@ncnt=\@ne
\newcount\sp@ncnt@ \sp@ncnt@=\@ne
\newdimen\@rrwd \newdimen\@rrdp

 \def\adjustbot#1{\option{\advance\@rrdp#1\relax}}
 
\def\pushvertex#1{\global\p@shlen#1\relax
   \global\let\maybepush=\dopush}

 \newdimen\p@shlen \p@shlen=\z@

 \let\maybepush=\relax
 \def\dopush{\ifinmeasureCD 
   \advance\locdimen by -\p@shlen 
   \else\advance \@rrwd by -\p@shlen \fi 
   \global\let\maybepush=\relax \global\p@shlen=\z@\relax}

 \def\span@ne{\global\sp@ncnt=\@ne\relax}
 \def\setsp@n#1#2{\global\sp@ncnt=#1\relax
   \ifx\relax#2\relax\else\global\sp@ncnt@=#2\relax\fi}

 \def\plnrhead{\llap{$\rightarrow\mkern-1.5mu$}}
 \def\plnlhead{\rlap{$\mkern-1.5mu\leftarrow$}}

 \def\clap#1{\hbox to \z@{\hss #1\hss}}

 \def\plndnhead{\hbox{\the\textfont3 \char'171}}
 \def\plnuphead{\hbox{\the\textfont3 \char'170}}
 \def\Dnhe@d{\hbox{\the\textfont3 \char'177}}
 \def\Uphe@d{\hbox{\the\textfont3 \char'176}}

 \def\ltxrhead{\raise\@xisheight
   \llap{\smash{\@linefnt\@getrarrow(1,0)}}}
 \def\ltxlhead{\raise\@xisheight
   \rlap{\@linefnt\@getlarrow(-1,0)}}
 \def\ltxuphead{\setbox\z@=\rlap{%
   \kern\@halfwidth\@linefnt\char'66}%
   \copy\z@\kern-\ht\z@}
 \def\ltxdnhead{\setbox\z@=\rlap{%
   \kern\@halfwidth\@linefnt\char'77}%
   \ht\z@=\z@\box\z@}

 \def\wd@sh{\kern0.5\CDdashlen}
 \def\hd@sh{\vrule height\topofshaft depth\botofshaft
    width\CDdashlen}
 \def\vd@sh{\hrule height\CDdashlen
   depth\z@ width\sh@ftdiam}

\def\xylist{14{3434}13{2414}12{1723}%
  23{1413}34{1153}11{0867}43{0707}%
  32{0580}21{0414}31{0291}41{0}}
\newcount\tgtcnt@
\def\find@xyargs{\dimen@=\@rrdp
  \advance\dimen@ by \CDstrutlen
  \tgtcnt@=\dimen@ \dimen@=\@rrwd 
  \divide\dimen@ by \@m 
  \divide \tgtcnt@ by \dimen@ 
  \expandafter\testxy\xylist\relax
  \unitlength=\@xarg\@rrdp
  \divide\unitlength by\@yarg\relax}
\def\testxy#1#2#3{\ifnum\tgtcnt@>#3
    \@xarg=#1\relax \@yarg=#2\relax
    \let\next=\ignorerest
  \else\let\next\testxy\fi\next}
\def\ignorerest#1\relax{\relax}

\let\scalefactor=\@ne
\def\SWarrow{\find@xyargs\vector
  (-\@xarg,-\@yarg)\scalefactor\hskip-\wd\@linechar}
\def\NWarrow{\find@xyargs\vector
  (-\@xarg,\@yarg)\scalefactor\hskip-\wd\@linechar}
\def\NEarrow{\find@xyargs\vector
  (\@xarg,\@yarg)\scalefactor}
\def\SEarrow{\find@xyargs\vector
  (\@xarg,-\@yarg)\scalefactor}
\def\rightupline{\find@xyargs\@linelen=\scalefactor
     \unitlength\@sline}
\def\rightdownline{\find@xyargs\@yarg=-\@yarg\relax
     \@linelen=\scalefactor\unitlength\@sline}

\def\Sim{\ifodd\row\setbox\z@=\hbox{$\sim$}\dimen@=\ht\z@
 \advance\dimen@ by -\@xisheight
  \vbox{\box\z@\kern-\@xisheight\kern\dimen@}%
  \else\hbox{$\wr$}\fi}

\def\harrow#1#2#3{\inmeasureCDtrue\findminarrwd
  {#2}{#3}{\sp@ncnt\minharrlen}\inmeasureCDfalse\span@ne
  \mathrel{\hbox{\options\hplace{#1}\ulabel{#2}\dlabel{#3}}}}

\def\noharrow{\harrow\hfill{}{}}
\def\vertexvarrow#1#2#3{\findarrdp \@rrwd=\z@ \setsp@n\@ne\@ne
  \vbox to \z@{\kern-1.2\CDstrutht
  \rlap{\options\vplace{#1}\llabel{#2}\rlabel{#3}}\vss}}

\newif\ifinmeasureCD
\def\measurelabel#1{\setbox\z@
  \hbox{$\scriptstyle#1\kern\labelsurr$}%
  \ifdim\wd\z@>\@rrwd \@rrwd=\wd\z@\fi}
\def\findminarrwd#1#2#3{\@rrwd=#3\relax
   \measurelabel{#1}\measurelabel{#2}}
\def\findCDarrwd#1#2{\@rrwd=\minCDharrlen
   \measurelabel{#1}\measurelabel{#2}%
  }

\newcount\row \row=\@ne \newcount\col \col=\@ne 
 \newcount\numrows 
\numrows=\@ne
 \newcount\numcols
\newcount\arrspan \newdimen\vrtxhalfwd  \newbox\tempbox

\def\DANABUG{\advance\col by \@ne
 \@rrwd=\minCDharrlen
  \advance\@rrwd by \vrtxhalfwd
  \advance\@rrwd by \CDarrsurr
  \ifnum\col>\numcols \numcols=\col
     \newlocdimen{col\the\col}\locdimen=\@rrwd 
  \else \ifdim\@rrwd>\c@l \c@l=\@rrwd\fi\fi}

\def\drop#1\\{
  \findvrtxhalfsum\DANABUG\advance\row by 2 \measureinit}

\def\measureinit{\col=\@ne \vrtxhalfwd=-\CDarrsurr\arrspan=\@ne\@rrwd=\z@
   \setbox\tempbox=\hbox\bgroup$}
\def\measure{
  \let\harrow\measureCDarrow
  \let\CDCR=\measureCR 
   \findminc@lwd 
  \inmeasureCDtrue
  \row=\@ne \numcols=\z@ \measureinit}

\def\endmeasure{\findvrtxhalfsum\DANABUG
  \numrows=\row 
  \inmeasureCDfalse}

\def\newlocdimen#1{\advance\dimenc@unt by \@ne
  \ifnum\dimenc@unt<\insc@unt
     \else\errmessage{No room for the CD}\fi
  \dimendef\locdimen=\dimenc@unt
  \expandafter\dimendef\csname#1\endcsname=\dimenc@unt}

 \def\r@wc@l{\csname row\the\row col\the\col\endcsname}
 \def\c@l{\csname col\the\col\endcsname}

 \def\findvrtxhalfsum{$\egroup
  \newlocdimen{row\the\row col\the\col}
  \locdimen=\vrtxhalfwd 
  \vrtxhalfwd=0.5\wd\tempbox 
  \advance\vrtxhalfwd by \CDarrsurr
  \advance\locdimen by \vrtxhalfwd 
  \advance\@rrwd by \locdimen 
  \maybepush
  \divide\@rrwd by \arrspan\relax
  \ifdim\@rrwd<\minc@lwd
    \ifnum\col>\@ne \@rrwd=\minc@lwd\fi \fi
  \loop 
    \ifnum\col>\numcols \numcols=\col
       \newlocdimen{col\the\col}
       \locdimen=\@rrwd 
    \else \ifdim\@rrwd>\c@l \c@l=\@rrwd\fi \fi
   \ifnum\arrspan>\@ne
      \advance\arrspan by -1 \advance\col by \@ne
  \repeat }

 \def\measureCDarrow#1#2#3{\findvrtxhalfsum
   \arrspan=\sp@ncnt\relax\global\sp@ncnt=1\relax
   \advance\col by \@ne
   \findCDarrwd{#2}{#3}%
   \setbox\tempbox=\hbox\bgroup$}

 \newcount\dr@tn \dr@tn=\z@
 \def\locate#1:#2{\ifinmeasureCD\else
   \count@=-#1
   \multiply\count@ by 2
   \advance\count@ by #2
   \dimen@=\count@\@rrwd
   \ifnum\dr@tn=\@ne\relax \else\dimen@=-\dimen@ \fi
   \dimen@i=\@rrdp
   \ifnum\dr@tn>\z@\advance\dimen@i by \CDstrutlen \fi
   \dimen@i=\count@\dimen@i
   \count@=#2 \multiply\count@ by 2
   \divide\dimen@ by \count@
   \divide\dimen@i by \count@
   \lift\dimen@i\nudge\dimen@\fi}

\def\betweenCDrows{\advance\row by \@ne \col=\@ne
\options}

\def\hbegin{\hbox\bgroup\kern\c@l \kern-\r@wc@l$}
\def\hend{$\glet\maybepush\relax \CDstrut\egroup}
\def\vbegin{\setbox\tempbox=\hbox\bgroup$}
\def\vend{$\egroup\ht\tempbox=\z@\dp\tempbox\CDvarrlen
  \box\tempbox}
\def\setCD{\let\harrow=\setCDarrow
  \let\CDCR=\setCR 
  \row=\@ne \col=\@ne \hbegin}
\let\endsetCD=\hend 

\def\findarrwd{\@rrwd=\z@ \count@=\col \advance\count@ by\sp@ncnt
  \loop\ifnum\count@>\col \advance\count@ by -1
      \advance\@rrwd by\csname col\the\count@\endcsname\repeat}
\def\setCDarrow#1#2#3{\kern\CDarrsurr\advance\col by \@ne
  \findarrwd \advance\@rrwd by -\r@wc@l  
  \@rrdp=\z@ 
  \maybepush
  \advance\col by -\@ne \advance\col by \sp@ncnt \span@ne
  \hbox to \@rrwd{\options
   \@rrwd=\scalefactor\@rrwd\hss
   \hplace{#1}\ulabel{#2}\dlabel{#3}\hss}%
   \kern\CDarrsurr}

\newdimen\labspacei 
\newdimen\labspaceii 

\newdimen\@xisheight
  \@xisheight=\the\fontdimen22\textfont2
\newdimen\labelskip
  \labelskip=\the\fontdimen10\textfont3 
\newdimen\topofshaft
\newdimen\botofshaft
\newdimen\botofulabel
\newdimen\topofdlabel
\def\getlabeldims{
  \topofshaft=0.5\sh@ftdiam
  \botofshaft=\topofshaft
  \advance\topofshaft by \@xisheight  
  \advance\botofshaft by -\@xisheight  
  \botofulabel=\topofshaft
  \advance\botofulabel by \labelskip
  \topofdlabel=\botofshaft
  \advance\topofdlabel by \labelskip}

\def\ulabel{\ifnum\row=\@ne\let\next\ulabeli
   \else\let\next\ulabellap\fi\next}
\def\ulabeli#1{\vbox{
  \clap{\kern-\@rrwd$\scriptstyle#1$}%
  \kern\botofulabel}\maybeoffset}
\def\ulabellap#1{\vbox to \z@{\vss
  \clap{\kern-\@rrwd$\scriptstyle#1$}%
  \kern\botofulabel}\maybeoffset}
\def\dlabel{\ifnum\row=\numrows\let\next\dlabeli
   \else\let\next\dlabellap\fi\next}
\def\dlabeli#1{\vtop{\kern\topofdlabel
  \clap{\kern-\@rrwd$\scriptstyle#1$}%
  }\maybeoffset}
\def\dlabellap#1{\vbox to \z@{\kern\topofdlabel
  \clap{\kern-\@rrwd$\scriptstyle#1$}%
  \vss}\maybeoffset}
\def\rlabel#1{\vbox to \z@{\vss
  \rlap{\kern\labelskip$\scriptstyle#1$}%
  \vss\kern-\@rrdp}\maybeoffset}
\def\llabel#1{\vbox to \z@{\vss
  \llap{$\scriptstyle#1$\kern\labelskip}%
  \vss\kern-\@rrdp}\maybeoffset}
\def\swlabel#1{\vtop{\kern0.5\@rrdp
  \llap{$\scriptstyle#1$\kern\labelskip\kern-0.5\@rrwd}
  }\maybeoffset}
\def\nwlabel#1{\vbox{
  \llap{$\scriptstyle#1$\kern\labelskip\kern-0.5\@rrwd}%
  \kern-0.5\@rrdp}\maybeoffset}
\def\selabel#1{\vtop{\kern0.5\@rrdp
  \rlap{\kern0.5\@rrwd\kern\labelskip$\scriptstyle#1$}%
  }\maybeoffset}
\def\nelabel#1{\vbox{
  \rlap{\kern0.5\@rrwd\kern\labelskip$\scriptstyle#1$}%
  \kern-0.5\@rrdp}\maybeoffset}
\def\cplace#1{\vbox to \z@{\vss
  \clap{$#1$\kern-\@rrwd}%
  \kern-\@rrdp\vss}\maybeoffset}
\def\hplace#1{\hbox to \@rrwd{#1}\maybeoffset}
\def\vplace#1{\clap{\vbox to \z@{#1\kern-\@rrdp}}\maybeoffset}

\newdimen\nudgeamount \nudgeamount=\z@
\newdimen\liftamount \liftamount=\z@
\let\maybeoffset\relax
\newbox\offsetbox \newdimen\lastheight
\def\dooffset{
  \setbox\offsetbox=\lastbox \lastheight=\ht\offsetbox 
  \setbox\offsetbox=\vbox{\kern-\liftamount\box\offsetbox}%
  \ht\offsetbox=\lastheight
  \kern\nudgeamount\box\offsetbox\kern-\nudgeamount
  \global\nudgeamount=\z@ \global\liftamount=\z@
  \glet\maybeoffset=\relax}
\def\nudge#1{\ifinmeasureCD\else
  \global\advance\nudgeamount#1\relax
  \global\let\maybeoffset\dooffset\fi}
\def\lift#1{\ifinmeasureCD\else
  \global\advance\liftamount#1\relax
  \global\let\maybeoffset\dooffset\fi}

\def\findarrdp{\@rrdp=\CDvarrlen
  \ifnum\sp@ncnt@>1
    \advance\@rrdp by \CDstrutlen
    \multiply\@rrdp by \sp@ncnt@
    \advance\@rrdp by -\CDstrutlen \fi
 }

\def\varrow#1#2#3{\ifnum\sp@ncnt>\@ne 
     \sp@ncnt@=\sp@ncnt\relax\fi
  \findarrdp \@rrwd=\z@ 
  \kern\c@l
   \hbox to \z@{\options
   \@rrdp=\scalefactor\@rrdp
    \hss\vplace{#1}\llabel{#2}\rlabel{#3}\hss}%
  \global\advance\col by \@ne \setsp@n\@ne\@ne
  }

\def\novarrow{\varrow\vfill{}{}}

\def\tweenarrows#1{\findarrwd \findarrdp \setsp@n\@ne\@ne
  \rlap{\options\cplace{#1}}}

\def\usarrow #1#2#3{\dr@tn=\@ne
  \findarrwd \findarrdp \setsp@n\@ne\@ne 
  \rlap{\options\cplace{#1}\nwlabel{#2}\selabel{#3}}%
  \dr@tn=\z@}
\def\dsarrow #1#2#3{\dr@tn=\tw@
  \findarrwd \findarrdp \setsp@n\@ne\@ne 
  \rlap{\options\cplace{#1}\swlabel{#2}\nelabel{#3}}%
  \dr@tn=\z@}
 \def\@rrow#1{\csname #1@rrow\endcsname}
 \def\R@rrow{\harrow \rtarrfill}
 \def\L@rrow{\harrow \ltarrfill}
 \def\V@rrow{\varrow \dnarrfill}
 \def\A@rrow{\varrow \uparrfill}
 \def\SE@rrow{\dsarrow \SEarrow}
 \def\NW@rrow{\dsarrow \NWarrow}
 \def\SW@rrow{\usarrow \SWarrow}
 \def\NE@rrow{\usarrow \NEarrow}
 \def\DS@rrow{\dsarrow \dnslope}
 \def\US@rrow{\usarrow \upslope}
 \def\upslope{\find@xyargs
       \@linelen=\unitlength\@sline}
 \def\dnslope{\find@xyargs\@yarg=-\@yarg\relax
       \@linelen=\unitlength\@sline}

\newtoks\optionlist 
\optionlist={}
\let\options\relax
\def\dooptions{\the\optionlist\global\optionlist={}%
  \glet\options=\relax}
\def\option#1{\ifinmeasureCD\else
  \glet\options=\dooptions
  \global\optionlist=\expandafter{\the\optionlist\relax#1}\fi}
\def\wider#1{\ifinmeasureCD\else
   \option{\advance\@rrwd by #1}\fi}
\def\deeper#1{\ifinmeasureCD\else
   \option{\advance\@rrdp by #1}\fi}

{\def\\{\global\let\sptoken= }\\ }

\def\CR{\futurelet\nexttok\testCR}
\def\testCR{\ifx\nexttok\sptoken
   \let\next\eatspaceCR\else\let\next\CDCR\fi\next}
\def\eatspaceCR#1 {\CR}
\def\measureCR{\ifx\nexttok\endmeasure\let\nextCR\relax
    \else\let\nextCR\drop\fi\nextCR}
\def\setCR{\ifodd\row
  \ifx\nexttok\endsetCD\else\hend\betweenCDrows\vbegin\fi
  \else\vend\betweenCDrows\hbegin\fi}

\countdef\dimenc@unt=11
\def\CD#1\endCD{
   \begingroup\let\\=\CR
  \m@th\offinterlineskip
   \measure#1\endmeasure\null\,\vcenter{\setCD#1\endsetCD}\,
   \endgroup
    }

\ifx\@clnwd\undefined \nol@g\else\catcode`\ =14\relax\fi
 \font\@linefnt=line10 
 \newcount\@tempcnta
 \newcount\@tempcntb
 \newdimen\@tempdima
 \newdimen\@tempdimb
 \newdimen\@wholewidth
 \newdimen\@halfwidth
   \@wholewidth\fontdimen8\@linefnt \@halfwidth .5\@wholewidth
 \newdimen\unitlength
 \newcount\@xarg
 \newcount\@yarg
 \newcount\@yyarg
 \newbox\@linechar
 \newdimen\@linelen
 \newdimen\@clnwd
 \newdimen\@clnht
 \newif\if@negarg
 
 \def\@whilenoop#1{}

 \def\@whiledim#1\do #2{\ifdim #1\relax#2\@iwhiledim{#1\relax#2}\fi}

 \def\@iwhiledim#1{\ifdim #1\let\@nextwhile=\@iwhiledim 
         \else\let\@nextwhile=\@whilenoop\fi\@nextwhile{#1}}

 \def\@sline{\ifnum\@xarg< 0 \@negargtrue \@xarg -\@xarg \@yyarg -\@yarg
   \else \@negargfalse \@yyarg \@yarg \fi
 \ifnum \@yyarg >0 \@tempcnta\@yyarg \else \@tempcnta -\@yyarg \fi
 \ifnum\@tempcnta>6 \@badlinearg\@tempcnta0 \fi
 \ifnum\@xarg>6 \@badlinearg\@xarg 1 \fi
 \setbox\@linechar\hbox{\@linefnt\@getlinechar(\@xarg,\@yyarg)}%
 \ifnum \@yarg >0 \let\@upordown\raise \@clnht\z@
    \else\let\@upordown\lower \@clnht \ht\@linechar\fi
 \@clnwd=\wd\@linechar
 \if@negarg \hskip -\wd\@linechar \def\@tempa{\hskip -2\wd\@linechar}\else
      \let\@tempa\relax \fi
 \@whiledim \@clnwd <\@linelen \do
   {\@upordown\@clnht\copy\@linechar
    \@tempa
    \advance\@clnht \ht\@linechar
    \advance\@clnwd \wd\@linechar}%
 \advance\@clnht -\ht\@linechar
 \advance\@clnwd -\wd\@linechar
 \@tempdima\@linelen\advance\@tempdima -\@clnwd
 \@tempdimb\@tempdima\advance\@tempdimb -\wd\@linechar
 \if@negarg \hskip -\@tempdimb \else \hskip \@tempdimb \fi
 \multiply\@tempdima \@m
 \@tempcnta \@tempdima \@tempdima \wd\@linechar \divide\@tempcnta \@tempdima
 \@tempdima \ht\@linechar \multiply\@tempdima \@tempcnta
 \divide\@tempdima \@m
 \advance\@clnht \@tempdima
 \ifdim \@linelen <\wd\@linechar
    \hskip \wd\@linechar
   \else\@upordown\@clnht\copy\@linechar\fi}
 
 \def\@getlinechar(#1,#2){\@tempcnta#1\relax\multiply\@tempcnta 8
 \advance\@tempcnta -9 \ifnum #2>0 \advance\@tempcnta #2\relax\else
 \advance\@tempcnta -#2\relax\advance\@tempcnta 64 \fi
 \char\@tempcnta}
 
 \def\vector(#1,#2)#3{\@xarg #1\relax \@yarg #2\relax
 \@tempcnta \ifnum\@xarg<0 -\@xarg\else\@xarg\fi
 \ifnum\@tempcnta<5\relax
 \@linelen=#3\unitlength
 \ifnum\@xarg =0 \@vvector 
   \else \ifnum\@yarg =0 \@hvector \else \@svector\fi
 \fi
 \else\@badlinearg\fi}
 
 \def\@svector{\@sline
 \@tempcnta\@yarg \ifnum\@tempcnta <0 \@tempcnta=-\@tempcnta\fi
 \ifnum\@tempcnta <5
   \hskip -\wd\@linechar
   \@upordown\@clnht \hbox{\@linefnt  \if@negarg 
   \@getlarrow(\@xarg,\@yyarg) \else \@getrarrow(\@xarg,\@yyarg) \fi}%
 \else\@badlinearg\fi}
 
 \def\@getlarrow(#1,#2){\ifnum #2 =\z@ \@tempcnta='33\else
 \@tempcnta=#1\relax\multiply\@tempcnta \sixt@@n \advance\@tempcnta
 -9 \@tempcntb=#2\relax\multiply\@tempcntb \tw@
 \ifnum \@tempcntb >0 \advance\@tempcnta \@tempcntb\relax
 \else\advance\@tempcnta -\@tempcntb\advance\@tempcnta 64
 \fi\fi\char\@tempcnta}
 
 \def\@getrarrow(#1,#2){\@tempcntb=#2\relax
 \ifnum\@tempcntb < 0 \@tempcntb=-\@tempcntb\relax\fi
 \ifcase \@tempcntb\relax \@tempcnta='55 \or 
 \ifnum #1<3 \@tempcnta=#1\relax\multiply\@tempcnta
 24 \advance\@tempcnta -6 \else \ifnum #1=3 \@tempcnta=49
 \else\@tempcnta=58 \fi\fi\or 
 \ifnum #1<3 \@tempcnta=#1\relax\multiply\@tempcnta
 24 \advance\@tempcnta -3 \else \@tempcnta=51\fi\or 
 \@tempcnta=#1\relax\multiply\@tempcnta
 \sixt@@n \advance\@tempcnta -\tw@ \else
 \@tempcnta=#1\relax\multiply\@tempcnta
 \sixt@@n \advance\@tempcnta 7 \fi\ifnum #2<0 \advance\@tempcnta 64 \fi
 \char\@tempcnta}
\catcode`\ =10

\dol@g 
\catcode`\@=\active
\LaTeXfeatures

\newcommand{\F}{{\mathbb F}}

\newcommand{\Q}{{\mathbb Q}}

\numberwithin{equation}{section}

\DeclareMathOperator{\Mod}{mod}

\begin{document}

\newtheorem{theorem}{Theorem}[section]
\newtheorem{lemma}[theorem]{Lemma}
\newtheorem{prop}[theorem]{Proposition}
\newtheorem{proposition}[theorem]{Proposition}
\newtheorem{corollary}[theorem]{Corollary}
\newtheorem{corol}[theorem]{Corollary}
\newtheorem{conj}[theorem]{Conjecture}

\theoremstyle{definition}
\newtheorem{defn}[theorem]{Definition}
\newtheorem{example}[theorem]{Example}
\newtheorem{examples}[theorem]{Examples}
\newtheorem{remarks}[theorem]{Remarks}
\newtheorem{remark}[theorem]{Remark}
\newtheorem{algorithm}[theorem]{Algorithm}
\newtheorem{question}[theorem]{Question}
\newtheorem{problem}[theorem]{Problem}
\newtheorem{subsec}[theorem]{}
\newtheorem{acknowledgements}[theorem]{Acknowledgements \nonumber}

\def\toeq{{\stackrel{\sim}{\longrightarrow}}}
\def\into{{\hookrightarrow}}


\def\alp{{\alpha}}  \def\bet{{\beta}} \def\gam{{\gamma}}
 \def\del{{\delta}}
\def\eps{{\varepsilon}}
\def\kap{{\kappa}}                   \def\Chi{\text{X}}
\def\lam{{\lambda}}
 \def\sig{{\sigma}}  \def\vphi{{\varphi}} \def\om{{\omega}}
\def\Gam{{\Gamma}}   \def\Del{{\Delta}}
\def\Sig{{\Sigma}}   \def\Om{{\Omega}}
\def\ups{{\upsilon}}


\def\F{{\mathbb{F}}}
\def\BF{{\mathbb{F}}}
\def\BN{{\mathbb{N}}}
\def\Q{{\mathbb{Q}}}
\def\Ql{{\overline{\Q }_{\ell }}}
\def\CC{{\mathbb{C}}}
\def\R{{\mathbb R}}
\def\V{{\mathbf V}}
\def\D{{\mathbf D}}
\def\BZ{{\mathbb Z}}
\def\K{{\mathbf K}}
\def\XX{\mathbf{X}^*}
\def\xx{\mathbf{X}_*}

\def\AA{\Bbb A}
\def\BA{\mathbb A}
\def\HH{\mathbb H}
\def\PP{\Bbb P}

\def\Gm{{{\mathbb G}_{\textrm{m}}}}
\def\Gmk{{{\mathbb G}_{\textrm m,k}}}
\def\GmL{{\mathbb G_{{\textrm m},L}}}
\def\Ga{{{\mathbb G}_a}}

\def\Fb{{\overline{\F }}}
\def\Kb{{\overline K}}
\def\Yb{{\overline Y}}
\def\Xb{{\overline X}}
\def\Tb{{\overline T}}
\def\Bb{{\overline B}}
\def\Gb{{\bar{G}}}
\def\Ub{{\overline U}}
\def\Vb{{\overline V}}
\def\Hb{{\bar{H}}}
\def\kb{{\bar{k}}}

\def\Th{{\hat T}}
\def\Bh{{\hat B}}
\def\Gh{{\hat G}}

\def\cF{{\mathfrak{F}}}
\def\cC{{\mathcal C}}
\def\cU{{\mathcal U}}

\def\Xt{{\widetilde X}}
\def\Gt{{\widetilde G}}

\def\gg{{\mathfrak g}}
\def\hh{{\mathfrak h}}
\def\lie{\mathfrak a}

\def\GL{\textrm{GL}}            \def\Stab{\textrm{Stab}}
\def\Gal{\textrm{Gal}}          \def\Aut{\textrm{Aut\,}}
\def\Lie{\textrm{Lie\,}}        \def\Ext{\textrm{Ext}}
\def\PSL{\textrm{PSL}}          \def\SL{\textrm{SL}}
\def\loc{\textrm{loc}}
\def\coker{\textrm{coker\,}}    \def\Hom{\textrm{Hom}}
\def\im{\textrm{im\,}}           \def\int{\textrm{int}}
\def\inv{\textrm{inv}}           \def\can{\textrm{can}}
\def\Cl{\textrm{Cl}}
\def\Sz{\textrm{Sz}}
\def\ad{\textrm{ad\,}}
\def\SU{\textrm{SU}}
\def\PSL{\textrm{PSL}}
\def\PSU{\textrm{PSU}}
\def\rk{\textrm{rk}}
\def\PGL{\textrm{PGL}}
\def\Ker{\textrm{Ker}}
\def\Ob{\textrm{Ob}}
\def\Var{\textrm{Var}}
\def\poSet{\textrm{poSet}}
\def\Al{\textrm{Al}}
\def\Int{\textrm{Int}}
\def\Mod{\textrm{Mod}}
\def\Smg{\textrm{Smg}}
\def\ISmg{\textrm{ISmg}}
\def\Ass{\textrm{Ass}}
\def\Grp{\textrm{Grp}}
\def\Com{\textrm{Com}}
\def\Im{\textrm{Im}}
\def\Val{\textrm{Val}}
\def\LKer{\textrm{LKer}}
\def\Val{\textrm{Val}}
\def\Th{\textrm{Th}}
\def\Set{\textrm{Set}}
\def\Hal{\textrm{Hal}}
\def\Lat{\textrm{Lat}}
\def\LK{\textrm{LK}}

\def\tors{_\def{\textrm{tors}}}      \def\tor{^{\textrm{tor}}}
\def\red{^{\textrm{red}}}         \def\nt{^{\textrm{ssu}}}

\def\sss{^{\textrm{ss}}}          \def\uu{^{\textrm{u}}}
\def\mm{^{\textrm{m}}}
\def\tm{^\times}                  \def\mult{^{\textrm{mult}}}

\def\uss{^{\textrm{ssu}}}         \def\ssu{^{\textrm{ssu}}}
\def\comp{_{\textrm{c}}}
\def\ab{_{\textrm{ab}}}

\def\et{_{\textrm{\'et}}}
\def\nr{_{\textrm{nr}}}

\def\nil{_{\textrm{nil}}}
\def\sol{_{\textrm{sol}}}
\def\End{\textrm{End\,}}

\def\til{\;\widetilde{}\;}





\title[Decompositions and complexity  of linear automata]
 {Decompositions and complexity of linear automata}

\author[B.Plotkin,T.Plotkin ]{Boris Plotkin and Tatjana Plotkin}

\address{ T.Plotkin: Dept. of Computer Science, Bar-Ilan Univ., Ramat Gan, Israel}
\email{plotkin\copyright macs.biu.ac.il}
\address{B.Plotkin: Dept. of Mathematics, Hebrew Univ., Jerusalem, Israel}
\email{borisov\copyright math.huji.ac.il}


\begin{abstract}

The Krohn-Rhodes complexity theory for pure (without linearity) automata  is well-known.  This theory uses an operation of wreath product as a decomposition tool.
The main goal of the  paper is to introduce the notion of   complexity of linear automata. This notion is ultimately related with decompositions of linear automata. The study of these decompositions is the second objective of the paper.   In order to define complexity for linear automata,  we have to use three operations, namely, triangular product of linear automata, wreath product of pure automata and wreath product of a linear automaton with a pure one which returns a linear automaton. We define the complexity of a linear automaton as the minimal number of operations in the decompositions of the automaton into indecomposable components (atoms). This theory relies on the following parallelism between wreath and triangular products: both of them are terminal objects in the categories of cascade connections of automata. The wreath product is the terminal object in the Krohn-Rhodes theory for pure automata, while the triangular product provides the terminal object for the cascade connections of linear automata.
\end{abstract}

\maketitle


\medskip


\smallskip


\section{Introduction}\label{sec:intro}


A pure semigroup automaton is a three-sorted algebraic structure of the form
 $(A, \Gamma, B)$, where $\Gamma$ is a semigroup, $A$, $B$ are the sets, and the
axioms $a \circ \gamma_1\gamma_2 = (a \circ
\gamma_1)\circ\gamma_2$, $a \ast \gamma_1\gamma_2 = (a \circ
\gamma_1)\ast\gamma_2$ are fulfilled for the operations $\circ: A\times \Gamma\to A$ and $\ast: A\times \Gamma\to B$. The complexity theory of such automata is well known (see, for example \cite{KR1}, \cite{KSt},
\cite{K}). The basis of this theory constitutes the famous Krohn-Rhodes decomposition theory  of
pure (semigroup) automata (see \cite{KR},\cite{KRT},
\cite{E},\cite{RS}, \cite{A},\cite{Es},
 etc.). This theory has many generalizations (\cite{RS}, \cite{VW}, \cite{W}).

 Suppose that we have a (finite) semigroup automaton $\mathcal{A}$. The
Krohn-Rhodes theory  basically says that any semigroup automaton  $\mathcal{A}$
 can be built up by cascading simple group automata, which
divide $\mathcal{A}$, with certain trivial automata, the so-called
"flips-flops". We would prefer to say that the Krohn-Rhodes Theorem
allows to decompose any finite semigroup automaton into
indecomposable (irreducible) bricks, via the construction of cascade
connection of automata. The cascade connection of semigroup automata $\mathcal{A}_1$ and $\mathcal{A}_2$ is tightly related to wreath product of  automata. 
It can be seen that every cascade connection of the  automata is embedded into their wreath product. The wreath product construction
leads to the decomposition of pure automata and to the definition of Krohn-Rhodes complexity (KR-complexity) of the automaton. Recall that
\begin{defn}The Krohn–-Rhodes complexity (group complexity)  of a pure semigroup automaton $\mathcal{A}$  is the least number of group automata in the Krohn-Rhodes decomposition of $\mathcal{A}$.
\end{defn}
The main goal of the  paper is to introduce the notion of   complexity of a linear automaton (see Section \ref{sec:semigr_lin_atm} for the definition). This notion is ultimately related to the decompositions of linear automata. The study of these decompositions is also one of the main objectives of the paper.   In order to define complexity for linear automata,  we have to use three operations, namely, the operations of triangular product of linear automata, wreath product of pure automata and wreath product of a linear automaton with a pure one which returns a linear automaton.

The paper is organized as follows. Our main goal, that is  the definition of complexity of a linear automaton, is considered in Section \ref{sec:Complexity}. The preceding sections serve  the aim to get all notions ready for Section \ref{sec:Complexity}. In particular, Sections \ref{sec:semigr_atm} and \ref{sec:semigr_lin_atm} introduce the notions of a pure semigroup and a linear semigroup automaton, respectively.  Sections \ref{sec:casc_con_pure_atm} - \ref{sec:Decomp} 
are devoted to above-mentioned operations and to the relation of divisibility of automata. Namely, Section \ref{sec:casc_con_pure_atm} deals with the construction of cascade connection of pure automata, in Sections \ref{sec:LA_rep_semigr} and \ref{sec:CC_and_TP} cascade connection of linear automata is considered and, correspondingly, the operation of triangular product of linear automata is discussed, and in Section \ref{sec:WP_LA_PA} the  operation of wreath product of
 a linear automaton and a pure one is treated. Finally, in Section \ref{sec:Decomp} we recall the main decomposition theorems.

In order to make the idea which rules the way to the definition of linear automata complexity more transparent we shall sketch here our strategy.

{\bf Strategy.} Recall that
an automaton $\Lambda_1$ is a divisor of an automaton $\Lambda_2$ if  $\Lambda_1$ is 
a homomorphic image of a sub-automaton of $\Lambda_2$. Denote this relation as $\Lambda_1 | \Lambda_2$. 
A natural way to define the complexity of an automaton is to present this automaton as a divisor of some  product of indecomposable automata (indecomposable with respect to the chosen  operation).

Decomposition operations and decomposition process allows us to determine atoms (indecomposable elements) of the decomposition. Complexity of a linear automaton can be defined either as the minimal number of atoms in the decompositions or as  the minimal number of operations. 
So the whole point is to determine the decomposition operations for the case of linear automata, to find out how the atoms look like, and to define the decomposition process.

Suppose we deal with a linear automaton $\Lambda=(A,\Gamma,B)$ where $A$ and $B$ are finite dimensional vector spaces over a filed $K$ and $\Gamma$ is a finite semigroup. We shall take into account the linear nature of the automaton $\Lambda$, the Krohn-Rhodes complexity of $\Gamma$, and their interaction.

Step 1. {\it Linearity.} Here, the main role is played by the operation of triangular product of automata. There exists a canonical representation of  $\Lambda$ as a divisor of the triangular product of atoms $(A_i,\Gamma)$,  where each $(A_i,\Gamma)$ is a faithful irreducible representation.  {\it Warning}: note that a  linear representation $(A,\Gamma)$, such that $A$ has no invariant subspaces, is irreducible. This is a linear atom in the sense of action of the semigroup $\Gamma$ on $A$. However, it is not an atom of the full decomposition of $\Lambda$ yet, since we have not exhausted the decomposition possibilities using additional operations of the wreath product type and have not worked with $\Gamma$.


Step 2. {\it Compression.} It is known \cite{CP} that there is a correspondence between representations of arbitrary finite semigroups and of completely $0$-simple semigroups. It gives a reduction from arbitrary  $(A,\Gamma)$ to $(A,\Sigma)$, where $\Sigma$ is completely $0$-simple. We call such a process compression.

Step 3. {\it Wreath products.}
Besides decomposition of an automaton into the triangular product we use the construction of wreath products of the form $(A,\Gamma, B) wr (Y,\Gamma')$, where $Y$ is a set and $\Gamma'$ is either a semigroup or a group, acting on $Y$. The result of wreath product of such form is also a linear automaton. {\it Warning}: indecomposable (with respect to triangular product) automata admit further decomposition using wreath products.

Now we are in the position when 
the acting semigroup $\Gamma$ is completely $0$-simple, and one can represent $\Gamma$ as a divisor of the wreath product of a group and a "flip-flop". Combining this fact with the just defined wreath product of a linear automaton and a semigroup (pure) automaton we arrive to the decomposition of the linear atom $(A,\Gamma)$ into the wreath product of a linear group automaton and a "flip-flop".

Step 4. According to step 1 the linear automaton  $\Lambda$ is a divisor of a triangular product of irreducible representations of the semigroup $\Gamma$. Using steps 2 and 3 and some properties of the introduced operations one can get that  $\Lambda$ is a divisor of a triangular product of wreath products, where each factor is either an irreducible representation $(A_i,G_i)$ of a group $G_i$, or a "flip-flop" automaton.

Step 5. Since the group $G_i$ is not necessarily simple there is a room for  further reduction.
In order to finish the process one should apply the constructions of triangular product and wreath product of a linear and pure automata to the obtained irreducible group representation $(A_i,G_i)$. Then, the indecomposable factors will be linear group automata with acting simple group and "flip-flops" (trivial factors). It remains to calculate either the number of non-trivial factors, or the number of operations involved.


We do not prove new theorems in the paper, using some material from \cite{PGG} instead. The main emphasis is put on organization of a decomposition process which allows us to get the notion of the complexity of a linear automaton as an output.




\section{Pure semigroup automata}\label{sec:semigr_atm}

Recall, (see Introduction) that
a pure semigroup automaton is a triple
 $(A, \Gamma, B)$, where $\Gamma$ is a semigroup, $A$, $B$ are the sets, and the operations $\circ$ and $\ast$
are correlated with the operation in the semigroup:
$$a \circ \gamma_1 \gamma_2 = (a \circ \gamma_1) \circ \gamma_2,$$
$$a \ast \gamma_1 \gamma_2 = (a \circ \gamma_1) \ast \gamma_2.$$

Given sets $A$ and $B$, denote by $S_A$ the semigroup of transformations of the set $A$ and by $Fun(A,B)$ the set of mappings from $A$ to $B$. Consider the Cartesian product $S_{A,B} = S_A \times Fun(A,B)$. Here $S_{A,B}$ is a semigroup with respect to the multiplication: $(\sigma_1, \varphi_1)(\sigma_2, \varphi_2)= (\sigma_1 \sigma_2, \sigma_1 \varphi_2)$, $\sigma \in S_A$, $\varphi \in Fun(A,B)$. Define an automaton $(A,S_{A,B},B)$ by the rule: $a \circ (\sigma, \varphi) = a \sigma$, $a \ast (\sigma, \varphi) = a \varphi$. 
Any semigroup automaton $(A, \Gamma, B)$ is determined by a homomorphism $\Gamma \to S_{A,B}$. In this sense the automaton $(A,S_{A,B},B)$ is universal.


 Consider a pure automaton $(A,X,B)$, where  $X$ is a set. We have a mapping $X \to S_{A,B}$. Let $F(X)$ be the free semigroup over the set $X$. The initial mapping is extended up to a homomorphism $\mu: F(X) \to S_{A,B}$, which determines a semigroup automaton $(A, F(X), B)$. We can pass from $(A, F(X), B)$ to a faithful semigroup automaton $(A,\Gamma,B)$ where $\Gamma$ is a result of factorization of the semigroup $F(X)$ by the kernel of $\mu $. 
 So, a pure automaton $(A,X,B)$ gives rise to a faithful semigroup automaton $(A,\Gamma,B)$. This transition allows us to construct a decomposition theory  for pure automata (Krohn-Rhodes theory).

\section{Semigroup linear automata}\label{sec:semigr_lin_atm}
The same reasoning can be repeated for linear automata. Given linear spaces $A$ and $B$ over a field $P$, take a semigroup (associative algebra) of endomorphisms $End(A)$ and a space of homomorphisms $Hom(A,B)$. Take a Cartesian product $End(A,B)= End(A) \times Hom(A,B)$. As in Section \ref{sec:semigr_atm}, proceed from the multiplication $(\sigma_1, \varphi_1)(\sigma_2, \varphi_2)= (\sigma_1 \sigma_2, \sigma_1 \varphi_2)$ and get a semigroup $End(A,B)$. 
 We can repeat almost literally the material from the previous section  and get a semigroup linear automaton $(A,End(A,B),B)$, but we will make one more step and generalize the notion of a linear automaton.

We shall use the language of matrices. Represent an element $(\sigma, \varphi)$ of $End(A,B)$ by a matrix

\[
    \begin{pmatrix}
        \sigma &\varphi \\
0 & 0 \\
    \end{pmatrix}.
    \]

\noindent
Then,

\[
    \begin{pmatrix}
     \sigma_1&\varphi_1 \\
0 & 0 \\
    \end{pmatrix}
        \begin{pmatrix}
     \sigma_2&\varphi_2 \\
0 & 0 \\
    \end{pmatrix}
    \quad = \quad
    \begin{pmatrix}
    \sigma_1\sigma_2&\sigma_1\varphi_2\\
    0&0
    \end{pmatrix}
              \]


\noindent
This motivates the above definition of multiplication in the semigroup $End(A,B)$, as well as gives an opportunity to consider matrices of the form
\[
    \begin{pmatrix}
        \sigma &\varphi \\
0 & \sigma ' \\
    \end{pmatrix},
    \]
where $\sigma'\in End(B)$,  with the usual matrix multiplication


\[
    \begin{pmatrix}
     \sigma_1&\varphi_1 \\
0 & \sigma_1 ' \\
    \end{pmatrix}
        \begin{pmatrix}
     \sigma_2&\varphi_2 \\
0 & \sigma_2 ' \\
    \end{pmatrix}
    \quad = \quad
    \begin{pmatrix}
    \sigma_1\sigma_2&\sigma_1 \varphi_2+\varphi_1\sigma_2 '\\
    0&\sigma_1 '\sigma_2 '
    \end{pmatrix}
              \]

Denote this semigroup of triangular matrices by the same $End(A,B)$. The semigroup $End(A,B)$ acts also in $B$.  We represent the semigroup $End(A,B)$ in the matrix form:

\[
    End(A,B)= \begin{pmatrix}
    End (A)&Hom(A,B)\\
0 & End(B) \\
    \end{pmatrix}
                 \]

 \noindent
 It leads to the automaton $(A, End(A,B), B)$. Here

\[
    a \circ \begin{pmatrix}
     \sigma&\varphi \\
0 & \sigma ' \\
    \end{pmatrix}
    = a \sigma , \quad \quad a\ast
        \begin{pmatrix}
     \sigma&\varphi \\
0 & \sigma ' \\
    \end{pmatrix}
     =  a \varphi , \quad \quad b \circ
    \begin{pmatrix}
    \sigma &\varphi\\
    0&\sigma '
    \end{pmatrix}
     =  b \sigma ' .
              \]

\noindent
This example hints the definition of a semigroup linear automaton.

\begin{defn}\label{de:lin}
A linear automaton is a triple $(A,\Gamma, B)$, where $A$ and $B$ are  modules over $K$, and   $\Gamma$ is a semigroup acting in $A$, from $A$ to $B$ and in $B$. The corresponding operations of actions $\circ$,
$\cdot$ and $\ast$ should satisfy the conditions:

\begin{itemize}
\item[1. ]{$a\circ \gamma_1 \gamma_2 = (a\circ \gamma_1)\circ \gamma_2$,}

\item[2. ]{$a\ast \gamma_1 \gamma_2 = (a\circ \gamma_1)\ast \gamma_2 + (a\ast \gamma_1)\circ \gamma_2 $,}

\item[3. ]{$b\cdot \gamma_1 \gamma_2 =(b\cdot \gamma_1)\cdot \gamma_2$.}
\end{itemize}
\end{defn}
\noindent
Here $a \in A$, $b \in B$, $\gamma_1, \gamma_2 \in \Gamma$. The automaton $(A, End(A,B),B)$ is a universal semigroup linear automaton and each automaton $(A, \Gamma, B)$ is determined by a homomorphism $\mu:\Gamma \to End(A,B)$. Faithfulness of an automaton means faithfulness of a homomorphism $\mu$.  To a linear automaton $(A, X, B)$ corresponds an automaton $(A, F(X), B)$ and  a faithful automaton $(A, \Gamma, B)$.


\section{Cascade connection of pure automata}\label{sec:casc_con_pure_atm}
We start the topic of constructions in automata theory. Cascade connections discussed here generalize parallel and serial connections of pure automata.

Let (pure) automata $(A_1, X_1, B_1)$ and $(A_2, X_2, B_2)$ be given. Their cascade connection has the form $(A_1 \times A_2, X, B_1\times B_2)$, where $X$ is a set equipped by the mappings:
$$\alpha : X \times A_2 \to X_1, \ \ \beta : X \to X_2.$$ We set:
$$(a_1, a_2) \circ x = (a_1 \circ \alpha(x,a_2), a_2\circ \beta(x)),$$
$$(a_1, a_2) \ast x = (a_1 \ast \alpha(x,a_2), a_2\ast \beta(x)).$$

The obtained automaton $(A_1 \times A_2, X, B_1\times B_2)$ is the {\it cascade connection} of $(A_1, X_1, B_1)$ and $(A_2, X_2, B_2)$  determined by a triple $(X, \alpha, \beta)$.

The category of such triples with the given automata $(A_1, X_1, B_1)$ and $(A_2, X_2, B_2)$ is defined naturally. Its morphisms $\mu:(X, \alpha, \beta) \to (X', \alpha ', \beta ')$ are presented by commutative diagrams

$$
\CD
X\times A_2 @>\alpha >> X_1\\
@. @/SE/\mu// @AA\alpha' A\\
@. X'\times A_2\\
\endCD
\qquad
\CD
X @>\beta >>\ X_2\\
@. @/SE/\mu// @AA\beta' A\\
@.\ X',\\
\endCD
$$
where $\mu(x,a_2) = (\mu(x), a_2)$. The category of triples determines the category of all cascade connections of the given automata.




The same construction works  for semigroup automata. Given  automata
$(A_1,\Gamma_1,B_1)$ and $(A_2,\Gamma_2,B_2)$, consider the triples
$(\Gamma, \alpha, \beta)$. Here $\beta: \Gamma \to \Gamma_2$ is a
homomorphism of semigroups and $\alpha: \Gamma \times A_2 \to
\Gamma_1$ satisfies the condition $\alpha(\gamma_1 \gamma_2,a) =
\alpha(\gamma_1,a) \alpha (\gamma_2, a\circ \beta(\gamma_1))$,
$\gamma_1, \gamma_2 \in \Gamma$, $a \in A_2$. In this framework the
corresponding cascade connection $(A_1 \times A_2, \Gamma, B_1
\times B_2)$ is a semigroup automaton, and, besides,   there is a
 category of cascade connections for the given pair
of semigroup automata.



Such a category has a universal terminal object, called the {\it wreath product} of the given automata and denoted by
$$(A_1, \Gamma_1, B_1)\ wr \ (A_2, \Gamma_2, B_2).$$
By the definition of a terminal object, every cascade connection of the given automata is embedded into their wreath product.

This universal object can be constructed in the very transparent way. Take the semigroup $\Gamma_1^{A_2}$ whose elements are mappings $\bar \gamma_1 : A_2 \to \Gamma_1$, $\bar \gamma_1 (a_2) = \gamma_1 \in  \Gamma_1$. The semigroup $\Gamma_2$ acts in $\Gamma_1 ^{A_2}$ by the rule $(\bar \gamma_1 \circ \gamma_2)(a_2) = \bar \gamma_1 (a_2 \circ \gamma_2)$. Consider the semidirect product  $\Gamma = \Gamma_1^{A_2} \leftthreetimes  \Gamma_2$ with the multiplication  $(\bar \gamma_1 , \gamma_2)(\bar \gamma_1 ' , \gamma_2 ')= (\bar \gamma_1 \cdot (\bar \gamma_1 ' \circ \gamma_2), \gamma_2 \gamma_2 ')$. This is the wreath product of semigroups $\Gamma_1 \ wr^{A_2} \ \Gamma_2$.
Setting $\alpha((\bar \gamma_1, \gamma_2), a_2) = \bar \gamma_1 (a_2)$ we define $\alpha : \Gamma \times A_2 \to \Gamma_1$. Setting $\beta(\bar \gamma_1, \gamma_2)= \gamma_2$ we get $\beta:\Gamma \to \Gamma_2$. The necessary conditions for $(\Gamma, \alpha, \beta)$ to be the terminal object are checked and we come up with the automaton
$$(A_1 \times A_2, \Gamma_1 wr^{A_2} \Gamma_2, B_1 \times B_2) = (A_1, \Gamma_1, B_1) \ wr \ (A_2, \Gamma_2, B_2).$$

The wreath product construction works in the Krohn-Rhodes theory which leads to the decomposition of pure automata and to the definition of complexity of this decomposition.




\section{Linear automata and representations of semigroups}\label{sec:LA_rep_semigr}
We denote a representation of the semigroup $\Gamma$ in a $K$-module $A$ by $(A, \Gamma)$. The representations are treated as semi-automata with the single operation $\circ: A\times \Gamma\to A$ subject to the condition $a \circ \gamma_1 \gamma_2 = (a \circ \gamma_1)\circ \gamma_2$. This means that  $a \to a \circ \gamma$ is an endomorphism of  $A$.


Let $(A, \Gamma_1)$ and $(B, \Gamma_2)$ be faithful representations. Denote by $(A, \Gamma_1)\nabla (B, \Gamma_2) = (A\oplus B, \Gamma)$ the triangular product of the representations $(A, \Gamma_1)$ and $(B, \Gamma_2)$. Here the semigroup $\Gamma$ is the semigroup of triangular matrices
\[ \gamma =
    \begin{pmatrix}
        \gamma_1&\varphi \\
        0       & \gamma_2 \\
    \end{pmatrix}
    \]
\noindent
where $\gamma_1 \in \Gamma_1$, $\gamma_2 \in \Gamma_2$, $\varphi \in Hom(A,B)$, and the automaton operation $\circ$ in $(A, \Gamma_1)\nabla (B, \Gamma_2)$ is defined by  $(a+b)\circ \gamma = a \circ \gamma_1 + (a \varphi + b \gamma_2)$. The same semigroup $\Gamma$ determines an automaton $(A, \Gamma, B)$. Here $a \circ \gamma = a \gamma_1$, $b \cdot \gamma = b \gamma_2$, $a \ast \gamma = a \varphi$. However, not every automaton can be obtained from a triangular product of representations.



Let us define the notion of {\it cascade connection of representations}. Given faithful representations $(A, \Gamma_1)$ and $(B, \Gamma_2)$, their cascade connection has the form $(A\oplus B, \Gamma)$. Besides, the following two conditions should be fulfilled:

\begin{itemize}
\item[1. ]{The subspace $B$ in $A\oplus B$ is invariant in respect to $\Gamma$, and the faithful representation corresponding to $(B,\Gamma)$ should be isomorphic to $(B, \Gamma_2)$}.


\item[2. ]{In the representation $(A\oplus B / B, \Gamma)$ the corresponding faithful representation should be isomorphic to $(A, \Gamma_1)$.}

\end{itemize}
Let us present several examples of cascade connections.

\begin{itemize}
\item[1. ]{Parallel connection $(A, \Gamma_1) \times (B, \Gamma_2) = (A\oplus B , \Gamma_1 \times \Gamma_2)$.}

\item[2. ]{Triangular product $(A, \Gamma_1)\nabla(B, \Gamma_2)$.}

\item[3. ]{If $(A, \Gamma, B)$ is an automaton, then the corresponding representation $(A\oplus B, \Gamma)$ defined by $(a+b)\circ \gamma = a \circ \gamma_1 + (a \ast \gamma +b \cdot \gamma_2)$  is a cascade connection of faithful representations $(A, \Gamma_1)$ and $(B, \Gamma_2)$.}
\end{itemize}

Cascade connections of the representations $(A, \Gamma_1)$ and $(B, \Gamma_2)$ can be built as follows.  
Fix natural projections $\delta_1: A\oplus B \to B$ and $\delta_2: A\oplus B \to A$. Consider   triples of the form  $(\Gamma, \alpha, \beta)$ with homomorphisms of semigroups $\alpha: \Gamma \to \Gamma_1$ and $\beta: \Gamma \to \Gamma_2$. Every triple $(\Gamma, \alpha, \beta)$, satisfying  
$$\delta_2 \gamma = \delta_2 \gamma^ \beta, \ \ \gamma \delta_1 = \delta_1 \gamma^ \alpha,$$
for every $\gamma \in \Gamma$, determines a  cascade connection $(A\oplus B , \Gamma)$ of representations $(A, \Gamma_1)$ and $(B, \Gamma_2)$.

 The category of cascade connections of faithful representations $(A, \Gamma_1)$ and $(B, \Gamma_2)$ is defined by morphisms (homomorphisms) of representations. One can prove that this category has a terminal object which is the triangular product of representations
$$(A, \Gamma_1)\nabla (B, \Gamma_2).$$

\section{Cascade connections and the triangular product of automata}\label{sec:CC_and_TP}
 Let us define a cascade connection of  linear automata $\Lambda_1 = (A_1, \Gamma_1, B_1)$ and $\Lambda_2 = (A_2, \Gamma_2, B_2)$. Along with $\Lambda_1$ and $\Lambda_2$  we have  representations $(A_1, \Gamma_1)$, $(B_1, \Gamma_1)$, $(A_2, \Gamma_2)$, $(B_2, \Gamma_2)$. Proceed to faithful representations $(A_1, \Gamma_1^1)$, $(B_1, \Gamma_1^2)$, $(A_2, \Gamma_2^1)$, $(B_2, \Gamma_2^2)$. Take a cascade connection of $(A_1, \Gamma_1^1)$ and $(A_2, \Gamma_2^1)$ and denote it by $(A_1 \oplus A_2, \Sigma_1)$. Denote a cascade connection of $(B_1, \Gamma_1^2)$ and $(B_2, \Gamma_2^2)$ by $(B_1 \oplus B_2, \Sigma_2)$. Assume these cascade connections to be faithful and take a cascade connection $((A_1 \oplus A_2) \oplus (B_1 \oplus B_2), \Gamma)$ of these cascade connections. The corresponding automaton  $(A_1 \oplus A_2, \Gamma, B_1 \oplus B_2)$ is said to be a {\it cascade connection of automata} $\Lambda_1$ and $\Lambda_2$.


The category of cascade connections of representations gives rise to the category of cascade connections of automata.

\begin{defn}
The category of cascade connections of automata $\Lambda_1$ and $\Lambda_2$ possesses a universal  terminal object. This object is called the triangular product $\Lambda_1 \nabla \Lambda_2$ of the given automata $\Lambda_1$ and $\Lambda_2$. We use also the notation
$$(A_1 \oplus A_2, \Gamma, B_1 \oplus B_2)=(A_1, \Gamma_1, B_1) \nabla (A_2, \Gamma_2, B_2).$$
\end{defn}

For the explicit construction of  $\Lambda_1 \nabla \Lambda_2$  we shall use the language of matrices.
We shall define the semigroup $\Gamma$ and the corresponding actions.
Take first the triangular product of faithful representations $(A_1, \Gamma_1^1)$ and $(A_2, \Gamma_2^1)$. The corresponding matrices have the form
\[ \gamma =
    \begin{pmatrix}
        \gamma_1&\varphi \\
        0       & \gamma_2 \\
    \end{pmatrix}
    \]
\noindent
where $\gamma_1 \in \Gamma_1^1$,  $\gamma_2 \in \Gamma_2^1$, $\varphi \in Hom(A_1, A_2)$. Similarly, we take matrices
\[ \gamma '=
    \begin{pmatrix}
        \gamma_1 '&\varphi ' \\
        0       & \gamma_2 '\\
    \end{pmatrix}
    \]
\noindent
where $\gamma_1 ' \in \Gamma_1^2$,  $\gamma_2 '\in \Gamma_2^2$, $\varphi' \in Hom(B_1, B_2)$ for $B_1$ and $B_2$.

Taking again the triangular product of these two representations we get the semigroup $\Gamma$ of matrices of the form
\[
    \begin{pmatrix}
     \gamma_{1}&\varphi&\varphi_{13}&\varphi_{14}\\
0 &\gamma_2 &\varphi_{23} &\varphi_{24}\\
0&0 &\gamma_1 '&\varphi '\\
0 & 0 & 0&\gamma_2 '
    \end{pmatrix}
         \]

\noindent
Here the matrix
\[
    \begin{pmatrix}
        \varphi_{13}&\varphi _{14} \\
        \varphi_{23}& \varphi_{24}\\
    \end{pmatrix}
    \]
determines $Hom(A_1 \oplus A_2, B_1 \oplus B_2)$, and  $\varphi_{13}:A_1 \to B_1$, $\varphi_{14}:A_1 \to B_2$, $\varphi_{23}:A_2 \to B_1$ and $\varphi_{24}:A_2 \to B_2$. The semigroup $\Gamma$ and the related operations $\circ:(A_1 \oplus A_2)\times\Gamma\to (A_1 \oplus A_2)$, $\ast:(A_1 \oplus A_2)\times\Gamma\to (B_1 \oplus B_2)$, and $\cdot:(B_1 \oplus B_2)\times\Gamma\to (B_1 \oplus B_2)$ determine the {\it  triangular product} $\Lambda_1\nabla\Lambda_2=(A_1 \oplus A_2, \Gamma, B_1 \oplus B_2)$ of the given automata.

 The triangular product of automata  is an associative operation,  used for the decomposition of linear automata. Another necessary operation is  defined in the next section.
The notion of the triangular product of linear representations is considered in the books \cite{PV}, \cite{V1}.

\section{Wreath product of a linear automaton and a pure one}\label{sec:WP_LA_PA}
Let $\Lambda = (A, \Gamma, B)$ be a linear automaton and $\Psi =(X, \Sigma)$ a pure representation of the semigroup $\Sigma$ which is a pure semi-automaton. Our aim is to define their wreath product $\Lambda \ wr \ \Psi$ which is a linear automaton as well. First of all, linearize $\Psi$. This means that we take the linear semi-automaton $(KX,\Sigma)$, where $KX$ is the linear envelope of $X$ over the ground field $K$.
  Take further the wreath product of semigroups $\Gamma wr^X \Sigma=\Gamma^X \times \Sigma$ whose elements have the form $(\bar \gamma, \sigma)$, $\bar \gamma : X \to \Gamma$,  $\sigma \in \Sigma$. Define the linear automaton
$$\Lambda \ wr \ \Psi = (A \otimes KX, \Gamma \ wr^X \  \Sigma, B \otimes KX).$$
Here $\otimes$ denotes  tensor product of spaces over $K$. 
Operations $\circ$, $\cdot$ and $\ast$ are defined on generators. Let $a \in A$ and $b \in B$. We set:
$$(a \otimes x)\circ (\bar \gamma, \sigma) = (a \circ \bar \gamma(x)) \otimes x \sigma,$$
$$(b \otimes x)\cdot (\bar \gamma, \sigma) = ( b \cdot \bar \gamma(x)) \otimes x \sigma,$$
$$(a \otimes x)\ast (\bar \gamma, \sigma) = (a \ast \bar \gamma(x)) \otimes x \sigma.$$
This determines the wreath product $\Lambda \ wr \ \Psi $  of the linear automaton $\Lambda = (A, \Gamma, B)$  and the representation $\Psi =(X, \Sigma)$.

Define also a cascade connection of the linear automaton $(A,\Gamma_1,B)$ with a pure one $(X,\Gamma_2)$. Consider first the triples $(\Gamma, \alpha,\beta)$ where $\alpha: X\times \Gamma\to \Gamma_1$ is a map subject to the condition
 $$
 \alpha(x,\gamma_1\gamma_2)=\alpha(x,\gamma_1)\alpha(x\circ \beta(\gamma_1),\gamma_2),
 $$
 and $\beta$ is a homomorphism of semigroups $\Gamma\to \Gamma_2$.

Define the linear automaton $(A\otimes KX,\Gamma, B\otimes KX)$ by the following rules: 
$$(a \otimes x)\circ \gamma   = (((a \circ \alpha(x, \gamma))\otimes (x\circ\beta(\gamma))), $$
$$(a \otimes x)\ast \gamma   = (((a \ast \alpha(x, \gamma))\otimes (x\circ\beta(\gamma))), $$
$$(b \otimes x)\cdot \gamma   = (((b \cdot \alpha(x, \gamma))\otimes (x\cdot\beta(\gamma))). $$
An  automaton of the type $(A\otimes KX,\Gamma, B\otimes KX)$ is called a {\it cascade connection of the linear automaton} $(A,\Gamma_1,B)$ {\it with the pure one} $(X,\Gamma_2)$. Consider the category of all cascade connections of this form. This category has the terminal object which is the wreath product of the linear automaton $(A,\Gamma_1,B)$ and the pure automaton $(X,\Gamma_2)$.

Proceed further from a semi-automaton $\Lambda = (A, \Gamma)$. Then \cite{PGG}, 

\begin{itemize}
\item[1. ]{$ (\Lambda \ wr \ \Psi_1)\ wr \ \Psi_2 = \Lambda \ wr \ (\Psi_1 \ wr \ \Psi_2),$}

\item[2. ]{$ \Lambda_1 \nabla (\Lambda_2 \ wr \ \Psi) \subset (\Lambda_1 \nabla \Lambda_2) \ wr \ \Psi,$}

\item[3. ]{$ (\Lambda_1 \nabla \Lambda_2) \ wr \ \Psi \subset (\Lambda_1 \ wr \ \Psi) \nabla (\Lambda_2 \ wr \ \Psi).$}
\end{itemize}


\section{Decomposition}\label{sec:Decomp}

Let $\mathfrak A_1$ and $\mathfrak A_2$ be two algebras in an arbitrary variety of algebras (multi-sorted in general). The following definition is well-known.
 \begin{defn} An algebra
 $\mathfrak A_1$ is called a divisor of $\mathfrak A_2$ if $\mathfrak A_1$ is 
 a homomorphic image of a subalgebra in $\mathfrak A_2$. We denote  $\mathfrak A_1 | \mathfrak A_2$.
 \end{defn}

In our situation if $\Lambda_1 | \Lambda_2$, $\Psi_1 | \Psi_2$, then $(\Lambda_1 \ wr \ \Psi_1) | (\Lambda_2 \ wr \ \Psi_2)$. This relation works in the decomposition theory of   linear automata. It is evident that $\Lambda_1 | \Lambda_2 \ \& \  \Lambda_2 | \Lambda_3$ implies $\Lambda_1 | \Lambda_3$.


\begin{defn}
A linear automaton $\Lambda_1 = \nabla_i (\Lambda_i \ wr \ \Phi_i)$ determines  a correct decomposition of the automaton $\Lambda$
 if $\Lambda | \Lambda_1$, and $\Lambda_i | \Lambda$ for all $i$. We denote a correct decomposition by $\Lambda \cong \Lambda_1$.
\end{defn}

\begin{lemma}\label{le:tr}
Correct decompositions  are transitive in the following sense. Let $\Lambda \cong \Lambda_1= \nabla_i (\Lambda_i \ wr \ \Phi_i)$,  and  $\Lambda_i \cong \Lambda_i '=\nabla_j(\Lambda_{ij} \ wr \ \Phi_{ij})  $   for each $i$. Define
$\Lambda_2 = \nabla_{ij} ((\Lambda_{ij} \ wr \ \Phi_{ij}) \ wr \ \Phi_i)$. Then $\Lambda \cong \Lambda_2$.
\end{lemma}
\begin{proof}
Assume  that $\Lambda \cong \Lambda_1= \nabla_i (\Lambda_i \ wr \ \Phi_i)$,  and that $\Lambda_i \cong \Lambda_i '=\nabla_j(\Lambda_{ij} \ wr \ \Phi_{ij})  $   for each $i$. Substitute $\Lambda_i$ by $\Lambda_i '$ in $\Lambda_1 = \nabla_i (\Lambda_i \ wr \ \Phi_i)$. We get
$$\Lambda_2 = \nabla_{ij} ((\Lambda_{ij} \ wr \ \Phi_{ij}) \ wr \ \Phi_i) = \nabla_{ij} (\Lambda_{ij} \ wr \ (\Phi_{ij} \ wr \ \Phi_i)).  $$
By the definition, we have the following properties:
$$\Lambda_1 | \Lambda_1 ' \ \& \  \Lambda_2 | \Lambda_2 ' \ \ {\rm implies}\ \ \Lambda_1 \nabla \Lambda_2| \Lambda_1 ' \nabla \Lambda_2 ',$$
$$\Lambda_1 | \Lambda ' \ \& \  \Phi | \Phi ' \ \ {\rm implies} \ \ \Lambda \ wr \ \Phi | \Lambda ' \ wr \ \Phi '.$$
These properties are valid for any number of factors.

Since $\Lambda_i \cong \Lambda_i '$, we have $\Lambda_i \ wr \ \Phi_i | \Lambda_i ' \ wr \ \Phi_i$. Since $\Lambda_i '=\nabla_j(\Lambda_{ij} \ wr \ \Phi_{ij})  $,  we get $\Lambda_i \ wr \ \Phi_i | \nabla_j (\Lambda_{ij} \ wr \ \Phi_{ij}) \ wr \ \Phi_i$ which gives also $\Lambda_i \ wr \ \Phi_i | \nabla_j (\Lambda_{ij} \ wr \ (\Phi_{ij} \ wr \ \Phi_i))$. Take triangular products by $i$ and get $\Lambda_1 | \Lambda_2$. Since $\Lambda | \Lambda_1$, then $\Lambda | \Lambda_2$. We have also $\Lambda_{ij} | \Lambda_i$ and $\Lambda_i | \Lambda$. This gives $\Lambda_{ij} | \Lambda$, which leads to the correct decomposition $\Lambda \cong \Lambda_2$.
\end{proof}



We shall use some results from \cite{PGG}. Consider, first, the linear decomposition of automata.

\begin{defn}
An automaton $\Lambda$ is linearly decomposable  into automata $\Lambda_i$,  if $\Lambda_i | \Lambda$ and $ \Lambda | \nabla_i \Lambda_i$, where $i=1, \ldots ,k$. 
\end{defn}
The idea of linear complexity relies on the next two theorems which deal with linear decomposition of automata.  Given a faithful automaton $\Lambda = (A, \Gamma, B)$, take subspaces $A_0$ in $A$ and $B_0$ in $B$ invariant under $\Gamma$. 
 Suppose that $a \ast \gamma \in B_0$ holds for every $a \in A_0$ and $\gamma \in \Gamma$.
 This leads to automata $(A_0, \Gamma, B_0)$ and $(A / A_0, \Gamma, B / B_0)$. Let us pass to faithful automata $\Lambda_1 = (A_0, \Gamma_1, B_0)$ and $\Lambda_2 = (A / A_0, \Gamma_2, B / B_0)$. Both automata $\Lambda_1$ and $\Lambda_2$ are divisors of the automaton $\Lambda$.
The following theorem holds:

\begin{theorem}\label{th:divisor_of_TP}
An automaton $\Lambda$ is a divisor of a triangular product $\Lambda_1 \nabla \Lambda_2$.
\end{theorem}

Thus, $\Lambda$ is linearly decomposable into a triangular product $\Lambda_1$ and $\Lambda_2$.
One can prove also that an automaton $(A, \Gamma, B)$ is indecomposable into a triangular product (i.e., linearly indecomposable)  if and only if either this automaton is a semi-automaton $(A, \Gamma, 0)$ and the representation $(A, \Gamma)$ is irreducible, or it is $(0, \Gamma, B)$ and the representation $(B, \Gamma)$ is irreducible.

Let, further, $A$ and $B$ have $\Gamma$-composition series of lengthes $n$ and $m$, respectively.

\begin{theorem}\label{th:factors}
An automaton $(A, \Gamma, B)$ is linearly decomposable into $n+m$ indecomposable factors. All of them are semi-automata (representations).
\end{theorem}

The number $n+m$ is an invariant of the automaton. We can consider this number  as a measure of the linear complexity of an automaton. In particular, if $A$ and $B$ are finitely dimension spaces, then there are finite compositional series and Theorem \ref{th:factors} is applicable. Let, for example, $\Gamma$ act triangularly in $A$ and $B$ and $n$ and $m$ be dimensions of $A$ and $B$ respectively. Then the linear complexity of the automaton is $n+m$.

\begin{remark}

In the papers \cite{K}, \cite{KSt} the triangular (and unitriangular) matrix semigroups over a finite field are considered. These semigroups are finite and in \cite{K} the Krohn-Rhodes complexity of these semigroups are calculated.
The complexity does not depend on a ground field and equals $n$ for the matrices of rank $(n+1)$. This means that we consider the semigroup as an abstract semigroup and  use the operation of wreath product of semigroups for a decomposition. 
\end{remark}


In view of Theorem \ref{th:factors}, from here on we will treat solely semi-automata. Let $(A, \Gamma)$ be an irreducible semi-automaton. It is linearly  indecomposable but one may consider its further decomposition using wreath product. We need here some semigroup background \cite{CP}.

 A subset $H$ in $\Gamma$ is an ideal if it is invariant under left and right multiplications in  $\Gamma$. Consider semigroups with zero. Such a semigroup $\Gamma$ is called $0$-simple if $\Gamma$ has no ideals except $0$ and  $\Gamma$ itself and $\gamma ^2 \neq 0$ for some $\gamma\in \Gamma$. A semigroup $\Gamma$ is completely $0$-simple if it is $0$-simple and contains a primitive idempotent. If $\Gamma$ is finite, then there is no difference between $0$-simple and completely $0$-simple semigroups. If $\Gamma$ is a semigroup without zero and without proper ideals, then, adding zero, we get a completely $0$-simple semigroup. In particular, if $\Gamma$ is a group, then, adding zero, we get a completely $0$-simple semigroup.


According to the well-known Rees theory \cite{CP}, each completely 0-simple semigroup has a special Rees matrix representation $\Gamma = (X,G,Y,[X,Y])$ where $X$ and $Y$ are sets, $G$ is a group with formally added zero and $P=[X,Y]$ is a sandwich matrix. To every pair $x$ and $y$, $x \in X$, $y \in Y$, it corresponds an element $g=[y,x]\in G$ which is also an element of the matrix $P$. Elements of $\Gamma$ are represented as triples $(x,g,y)$ and multiplication is defined by the rule $(x_1, g_1, y_1) (x_2, g_2, y_2) = (x_1, g_1[y_1, x_2] g_2, y_2)$.

Take a set $M$ and relate to it two semigroups $M^l$ and $M^r$. In $M^l$ we have $m_1m_2=m_1$ for $m_1, m_2 \in M$, in $M^r$ analogously: $m_1m_2=m_2$. We
will use the pure semi-automaton $(M,M^r)$, $m_1\circ m_2=m_1m_2=m_2$  naturally arising here.

Let $\Gamma$ be a Rees matrix semigroup, $G$ a corresponding group and $(A_1, \Gamma)$ an irreducible representation. Then there exists  an irreducible representation $(A, G)$ which is a divisor of $(A_1, \Gamma)$ and:

\begin{theorem} \label{th:divisor} \cite{PGG}
In the pointed conditions we have
$$
(A_1, \Gamma)| (A,G) \ wr \ (Y,Y^r)
$$
for some $Y$.
\end{theorem}

Since irreducible semi-automata are the atoms of the triangular (linear) decomposition of a linear automaton $\Lambda=(A,\Gamma, B)$, this theorem yields a further decomposition of $\Lambda$ for the case $\Gamma$ being a completely $0$-simple semigroup.

We recall a reduction of the general case to the situation of completely $0$-simple semigroups \cite{CP}. 
Let $(A, \Gamma)$ be an irreducible representation with a finite semigroup $\Gamma$. Take a two-sided ideal $U$ in $\Gamma$, consisting of elements of $\Gamma$ acting in $A$ as zero, i.e., $a \circ \gamma=0$ for $\gamma \in U$ and each $a \in A$. Let, further, $V\neq U$ be a minimal two-sided ideal in $\Gamma$ containing $U$.
 There arises a representation $(A,V)$. Prove that it is irreducible as well. Pass to a semigroup algebra $KV$. Together with $(A,V)$ we have a representation $(A,KV)$. For  a nonzero element $a$ in $A$ consider a subspace $a \circ KV$. It is invariant under $V$. Since $V$ is a two-sided ideal, the space is invariant under $\Gamma$. Hence,  $a \circ KV=A$ or $a \circ KV=0$. Consider all possible $a \in A$ with $a \circ KV=0$ and let $A_0$ be a subspace generated by all such elements $a$. $A_0$ is also invariant under $\Gamma$, namely, $(a \circ \gamma)\circ v=a\circ\gamma v=0$ for $\gamma \in \Gamma$, $v \in V$ and $a \in A_0$. So, $A_0 =0$ or $A_0 =A$. If $A_0=0$ or $a \circ KV=0$, then $a \circ KV=0$ implies $a \in A_0$ and $a=0$. We proceed from a nonzero $a$ and, thus, $a \circ KV\neq 0$, i.e., $a \circ KV=A$.

Let now $A_0=A$. Then $A$ is generated by elements $a$ with $a \circ KV=0$. Therefore, $a \circ KV=0$ for any $a \in A$. In particular, $a \circ v=0$ always holds true. This contradicts the condition $U \neq V$. Hence, the case $A_0=A$ is impossible.

So,  $a \circ KV=A$, we have an irreducible representation $(A,V)$. 
Take a semigroup $\Sigma = V / U$. It is clear that $\Sigma$ acts in $A$, it acts irreducibly, this semigroup is 0-simple and it is finite. We say that an irreducible representation $(A, \Gamma)$ with the finite semigroup $\Gamma$ is {\it compressed} into an irreducible representation $(A,\Sigma)$ with a completely 0-simple finite semigroup $\Sigma$.  Clearly,  $(A,\Sigma) | (A, \Gamma)$, and thus, $(A,G) | (A, \Gamma)$.

The introduced construction allows us to reduce the general theory with the finite $\Gamma$ to the situation of completely 0-simple $\Gamma$.

The next step of the decomposition process is the theorem \cite{PGG}  which deals with irreducible representations $(A,G)$ with the finite group $G$.

\begin{theorem}\label{th:irred_rep}
Let $(A,G)$ be an irreducible finitely dimension representation of a finite group $G$. Then
$$
(A,G) | (\nabla_i (A_i, H)) \ wr \ (X, \Phi),$$
\noindent
where all $(A_i,H)$ are irreducible representations of a simple group $H$, all of them are divisors of the representation $(A,G)$ and $(X,\Phi)$ is a representation of the group $\Phi$ on the set $X$. Everything is finite.
\end{theorem}
\begin{defn}
A semi-automaton $(A,\Gamma)$ with a finite semigroup $\Gamma$ is called decomposable if it is decomposable with respect to triangular products and  wreath products with a pure one.
\end{defn}

\begin{theorem}\label{th:indecomposable_semi-atm}
A semi-automaton $(A,\Gamma)$ is indecomposable  if and only if it is an irreducible  representation of a simple group $G$.
\end{theorem}



\section{Complexity}\label{sec:Complexity}
As we know, along with the absolutely non-decomposable linear automata  the indecomposable factors for pure automata  have the form $(X,G)$ with a simple group $G$, which acts transitively on $X$. Recall that in the linear case these are irreducible representations $(A,G)$ with a simple group $G$. All these $(A,G)$ and $(X,G)$ are atoms of the corresponding decompositions.

In the decomposition process we use several operations and constructions.  Let us underline once again that the following operations are used:

\begin{itemize}
\item[1) ]{Triangular product of linear automata.}

\item[2) ]{Wreath product of a linear automaton with a pure one.}

\item[3) ]{Wreath product of pure automata.}
\end{itemize}

Keeping in mind an arbitrary finite semigroup $\Gamma$, we use a compressing operation of an irreducible automaton $(A, \Gamma)$ to an irreducible automaton $(A, \Sigma)$ with 0-simple semigroup $\Sigma$.

\begin{defn}
We propose the complexity of an automaton to be a minimal number of the used operations in the decompositions of the automaton  into indecomposable components.
\end{defn}

\begin{remark} It is possible to define the complexity of a linear automaton to be the minimal number of linear atoms, i.e., the number of linear irreducible representations of simple groups in the decompositions of the automaton into indecomposable components. It is also possible to define the complexity (group complexity) of a linear automaton to be the minimal number of group atoms, i.e., the number of linear irreducible representations of simple groups in the decompositions of the automaton into indecomposable components plus the number  of transitive representations of simple groups in these decompositions.

We have chosen for the definition the minimal number of the operations used in decompositions because it corresponds to standard algorithmic approach to complexity. All three approaches give different numbers which are tightly related to each other.
\end{remark}

Let us outline the {\it general decomposition strategy}.  A given  linear automaton $\Lambda=(A, \Gamma, B)$ with a finite semigroup $\Gamma$ and finitely dimension $A$ and $B$ can be decomposed into indecomposable automata in respect to the triangular product of automata operation. These  indecomposable automata are irreducible representations of the form $(A, \Gamma)$. The number of such automata is determined by the lengthes of compositional series in respect to $\Gamma$ in $A$ and $B$. It is the sum of such lengthes, say, n. The corresponding number of operations $\nabla$ is $n-1$.

Now apply the compressing operation in every irreducible representation $(A, \Gamma)$ and pass to $(B, \Sigma)$ with 0-simple $\Sigma$. We add $n$ operations of compression and the total number of operations will be $n-1+n=2n-1$. For every such $\Sigma$ take the corresponding group $G$ (arising from the matrix representation $(X,G,[X,Y],Y)$ of $\Sigma$ ) and a corresponding set $A_1\subset A$. Now take an irreducible representation $(A_1, G)$ which is  a divisor of $(B, \Sigma)$. 
We have $(A, \Gamma) | (A_1, G)\ wr \ (Y, Y^r)$. We add $n$ wreath products, which sums up to $3n-1$ operations.

Further we decompose these representations $(A,G)$.  Take a composition series $G \supset H_1 \supset \ldots \supset H_{k-1} \supset 1$ in $G$. All factors here are simple groups. We have $(A, G) | (A, H_{k-1}) \ wr \ (X, \Phi)$ where $(X, \Phi)$ is $k-1$ pure wreath products of the form $(X_i, G_i)$ with all $G_i$ being simple groups. So, for each of $n$ representations there are added corresponding $k-1$ pure indecomposable wreath products. Now we use the well known fact that if $(A,G)$ is a completely reducible representation and $H$ is a normal subgroup in $G$ then the representation $(A,H)$ is also completely reducible.
This gives completely reducible decomposition of the representation $(A, H_{k-1})$ into irreducible components $(A_s, H_{k-1})$, and the number of such representations depends on the initial representation $(A, H_{k-1})$. To this decomposition it corresponds a triangular decomposition for $(A, H_{k-1})$ of the form $\nabla_s (A_s, H_{k-1})$. Each factor is an irreducible linear representation of a simple group.

So, in the series of decompositions we reach simple (indecomposable) components. The total number of the used operations can be computed. It is minimal, and it is the complexity of $\Lambda $.

Along with this method we could use another one with the greater number of operations. In particular, this could occur if we apply the rule
$$ (\Lambda_1\nabla \Lambda_2) \ wr \ \Psi \subset (\Lambda_1 \ wr \ \Psi) \nabla (\Lambda_2 \ wr \ \Psi) $$
on one of the previous steps. The number of operations increases.

Note also that in the calculations above we used the transitivity property for correct decomposition (see Lemma \ref{le:tr}).

Consider an example. Let an irreducible representation $(A,G)$ with the group $G$ having a simple invariant subgroup $H$ with a simple quotient group $G / H$ be given. Let us decompose it and calculate the complexity.
We have completely reducible representation $(A,H)$. Let $A = A_1 \oplus A_2 \oplus A_3$ be a corresponding decomposition, all $(A_i,H)$ being irreducible and absolutely simple. Denote $G_1 =G / H$ and take a set $G_1=X$. We have a simple representation $(X, G_1)$. For $(A,G)$ we have a decomposition
$$(A,G)|(A,H) \  wr \ (X,G_1).$$
Further,
$$(A,H)|(A_1,H) \nabla (A_2,H) \nabla (A_3,H)$$
and all $(A_i,H)$ are divisors of $(A,H)$. Finally, we have
$$(A,G)|((A_1,H) \nabla (A_2,H) \nabla (A_3,H))\ wr \ (X, G_1).$$
All $(A_i,H)$ are divisors of $(A,G)$, and the complexity is 3.

We have also
$$(A,G)|\nabla_i((A_i,H) \ wr \ (X,G_1))$$
where the number of operations is 5.



\end{document}